\documentclass[12pt,a4paper]{amsart}
\usepackage[latin1]{inputenc}
\usepackage{graphicx}
\usepackage{amssymb, amsmath}
\usepackage{geometry}
\usepackage{enumerate}
\usepackage[colorlinks=true,linkcolor=blue,urlcolor=blue,citecolor=blue]{hyperref}
\usepackage[all]{xy}
\usepackage{tikz-cd}

\newtheorem{theorem}{Theorem}[section]
\newtheorem{definition}[theorem]{Definition}
\newtheorem{lemma}[theorem]{Lemma}
\newtheorem{corollary}[theorem]{Corollary}
\newtheorem{proposition}[theorem]{Proposition}

\theoremstyle{definition}
\newtheorem{example}[theorem]{Example}
\newtheorem{remark}[theorem]{Remark}

\begin{document}
\setcounter{tocdepth}{1}

\title[]{The strong Nakano property in Banach lattices}

\author[Y. Azouzi]{Youssef Azouzi}
\address{Research Laboratory of Algebra, Topology, Arithmetic, and Order\\
Department of Mathematics\\
Faculty of Mathematical, Physical and Natural Sciences of Tunis\\
Tunis-El Manar University, 2092-El Manar, Tunisia}
\email{youssef.azouzi@ipest.rnu.tn}

\author[A. Ben Rjeb]{Asma~Ben Rjeb}
\address{Preparatory institute for engineering studies Nabeul\\
Research Laboratory of Algebra, Topology, Arithmetic, and Order\\
Department of Mathematics\\
Faculty of Mathematical, Physical and Natural Sciences of Tunis\\
Tunis-El Manar University, 2092-El Manar, Tunisia}
\email{asmabenrjab1992@gmail.com }

\author[P. Tradacete]{Pedro~Tradacete}
\address{Instituto de Ciencias Matem\'aticas (CSIC-UAM-UC3M-UCM)\\
Consejo Superior de Investigaciones Cient\'ificas\\
C/ Nicol\'as Cabrera, 13--15, Campus de Cantoblanco UAM\\
28049 Madrid, Spain.}
\email{pedro.tradacete@icmat.es}

\begin{abstract}
We study the strong Nakano property in Banach lattices with a special focus on free Banach lattices. We show that for every finite dimensional Banach space $E$, the free Banach lattice $FBL[E]$ has the strong Nakano property with a constant independent of the dimension. It is also shown that if $FBL[E]$ has the strong Nakano property, then $E$ cannot contain subspaces isomorphic to neither $c_0$ nor $L_1$.
\end{abstract}

\subjclass[2020]{46B42}

\keywords{Banach lattice; Strong Nakano propery; Free Banach lattice}

\maketitle

\section{Introduction and preliminaries}

In the Banach lattice setting there is a number of relations between norm boundedness of a set and existence of suprema or upper bounds with specific properties. Besides the classical notions of order (or Dedekind) completeness, one should recall that a Banach lattice $X$ is said to have the Fatou property if for every upwards directed subset $A \subset X_+$ with supremum $b$, then 
$$
\| b\| = \sup \{\|a\|:a\in A\}.
$$ 
A natural generalization of this notion was considered by H. Nakano \cite{Nakano, Nakano2} (and actually coined in \cite[Theorem 3]{Wickstead}): a Banach lattice $X$ has the Nakano property if for every upwards directed order bounded set $A\subset X_+$ one has 
$$
\inf\{\|b\|:b \text{ an upper bound for A}\} = \sup \{\|a\|:a\in A\}.
$$
The relation between these properties was clarified in \cite[Theorem 3]{Wickstead}: a Banach lattice $X$ has the Nakano property if and only if its Dedekind completion $\hat X$ has the Fatou property.

In this note we will study a stronger version of the Nakano property. Namely, a Banach lattice $X$ is said to have the strong Nakano property if for every norm-bounded upwards directed subset $A \subset X_+$ there is an upper bound $b\in X_+$ such that 
$$
\| b\| = \sup \{\|a\|:a\in A\}.
$$
This notion was introduced in \cite{ART}, where it was proved that the free Banach lattice $FBL(A)$ generated by any set $A$ has the strong Nakano property, thus answering a question from \cite{dePW}.

Naturally, one can also consider particular analogues of the strong Nakano property: we say that $X$ has the strong $\sigma$-Nakano property when every norm-bounded increasing \emph{sequence} $(x_n)\subset X_+$ has an upper bound $b\in X_+$ such that $\|b\|=\sup_n\|x_n\|$; also, for $\lambda\geq1$, we say that $X$ has the $\lambda$-strong Nakano property when every norm-bounded upwards directed set $A\subset X_+$ has an upper bound $b\in X$ with 
$$
\|b\|\leq \lambda \sup\{\|a\|:a\in A\};
$$
similarly, for $\lambda\geq1$, we will say that $X$ has the $(\lambda_+)$-strong Nakano property when every norm-bounded upwards directed set $A\subset X_+$ has for every $\lambda'>\lambda$ an upper bound $b\in X_+$ with 
$$
\|b\| \leq \lambda' \sup\{\|a\|:a\in A\}.
$$

In this paper we start with some basic properties of Banach lattices with the strong Nakano property and then focus our analysis on the case of free Banach lattices. To motivate this notion, let us briefly recall the construction of the free Banach lattice generated by a Banach space: Given a Banach space $E$, the \textit{free Banach lattice generated by $E$} will be a Banach lattice $FBL[E]$ together with a linear isometric embedding $\delta_E:E\rightarrow FBL[E]$, such that for every Banach lattice $X$ and every linear and bounded operator $T:E\rightarrow X$ there is a unique lattice homomorphism $\hat T$ making the following diagram commutative:
$$
\xymatrix{FBL[E]\ar@{-->}[rd]^{\hat{T}}&\\
     E\ar[r]^{T}\ar^{\delta_E}[u]& X}
$$
Moreover, $\|\hat{T}\|=\|T\|$.

The existence of $FBL[E]$, for every Banach space $E$, was proved in \cite{ART} via an explicit construction: For a positively homogeneous function $f:E^*\rightarrow \mathbb R$ consider the expression
\begin{equation}\label{eq: norm fbl}
	\|f\|_{FBL[E]}=\sup\Big\{\sum_{i=1}^m |f(x_i^*)|:m\in \mathbb N, (x_i^*)_{i=1}^m\subset E^*, \sup_{x\in B_E} \sum_{i=1}^m |x_i^*(x)|\leq 1\Big\},
\end{equation}
and define $H_1[E]$ to be the space of positively homogeneous functions $f:E^*\rightarrow \mathbb R$ such that$\|f\|_{FBL[E]}<\infty$, endowed with the pointwise order and lattice operations. It is easy to see that $H_1[E]$ is a Banach lattice with the norm $\|\cdot\|_{FBL[E]}$. 

Observe that for each $x\in E$, the functions $\delta_x:E^*\rightarrow \mathbb R$ given by $\delta_x(x^*)=x^*(x)$ for $x^*\in E^*$, can be used to define a linear isometry $x\mapsto \delta_E(x)=\delta_x$ from $E$ into $H_1[E]$. \cite[Theorem 2.5]{ART} shows that the closed sublattice generated by $\delta_E(E)$ in $H_1[E]$, provides an explicit description of $FBL[E]$, representing this as a space of functions on $E^*$. 

It is worth noting that the above construction defines a functor from the category of Banach spaces with bounded linear operators into that of Banach lattices with lattice homomorphisms, and has been the object of intense research in the last few years (see for instance \cite{AMRT, AMR2, DMRR2, GST, dHT, JLTTT} for different developments including relations with projectivity, duality, complex scalars and $p$-convex Banach lattices, among others). Certain properties of a Banach space can be identified with a corresponding analogue in Banach lattices via this functor (see in particular \cite[Section 10]{OTTT} for the basis of a dictionary between Banach space and Banach lattice properties).

The paper is organized as follows: In the next section, we begin with a systematic study of the strong Nakano property showing its relations to other classical properties of Banach lattices (such as existence of strong units, monotonically complete norms, or order continuity). We also provide some results about stability of the strong Nakano property under taking spaces of regular operators or sublattices and some facts about projections. In addition, it is observed that projective Banach lattices have the strong Nakano property. This is a consequence of the fact that the free Banach lattice $FBL(A)=FBL[\ell_1(A)]$ has the strong Nakano property for every set $A$ \cite[Theorem 4.11]{ART}. In particular, the latter implies that free Banach lattices generated by a finite dimensional space $E$ have the strong Nakano property with a constant depending on the Banach-Mazur distance between $E$ and $\ell_1(dim(E))$. This motivates the question on the optimality of this constant. In Section \ref{s:fbl}, using ideas from the local theory of Banach spaces about estimating summing norms of operators with few vectors, we will show that there is a universal constant $\lambda_{fin}$ such that for every finite-dimensional Banach space $E$, the space $FBL[E]$ has the $\lambda_{fin}$-strong Nakano property. We also show that Banach spaces $E$ for which $FBL[E]$ has the strong Nakano property cannot contain subspaces isomorphic to $L_1$ nor $c_0$. Finally, in Section \ref{s:fblpell1}, we provide the details to see that the free $p$-convex Banach lattice generated by $\ell_1(A)$, $FBL^{(p)}[\ell_1(A)]$ always has the strong Nakano property for $p\geq1$.

For background on Banach lattices we refer the reader to the monographs \cite{AB,LT2,M-N} and for the specifics of free Banach lattices to \cite{ART} and \cite{OTTT}.

\section{Basic facts}

It is clear that every Banach lattice with the strong Nakano property, has in particular the Nakano property, and this in turn implies the Fatou property. A simple example of a Banach lattice with the Fatou property but not the Nakano property is exhibited in \cite[Example 2]{Wickstead}. On the other hand, it is easy to check that $c_0$ has the Nakano, but not the $\lambda$-strong Nakano property for any $\lambda\geq1$ (this can be witnessed by the summing basis). In fact, we will see next that the only AM-spaces with the strong Nakano property are $C(K)$ spaces. Recall that an AM-space is a Banach lattice whose norm satisfies 
$$
\|x \vee y \|=\max\{\|x\|,\|y\|\},
$$
for every $x,y$ in $X_+$. A classical result due to S. Kakutani states that AM-spaces are always lattice isometric to a closed sublattice of some $C(K)$, whereas in the case they have strong unit (an element $e$ satisfying $\|x\|\leq1$ if and only if $|x|\leq e$), they have to be isometric to $C(K)$ for some compact Hausdorff space $K$ (cf. \cite[Theorem 1.b.6]{LT2}).

\begin{proposition}\label{p:AM}
Let $X$ be an AM-space. $X$ has the strong Nakano property if and only if $X$ has a strong unit.
\end{proposition}

\begin{proof}
Suppose $X$ is an AM-space with a strong unit. Then there exists a compact Hausdorff space $K$ such that $X$ is lattice isometric to the space $C(K)$ \cite[Theorem 4.29]{AB}. Since $C(K)$ has the strong Nakano for every compact $K$ (with constant functions playing the role of the corresponding upper bounds) the same holds for $X$.
 
Conversely, suppose an AM-space $X$ has the strong Nakano property. The set $B_+=\{x\in X_+:\|x\|\leq1\}$ is an upward directed norm bounded subset of $X_+$. Hence, there is  $e\in X_+$ an upper bound of $B_+$ such that 
$$
\|e\| = \sup \{\|b\|,b\in B_+\}=1.
$$
In particular, $|x| \le \|x\| e$ for every $x\in X$. Therefore, $e$ is a strong unit of $X$.
\end{proof}

A Banach lattice $X$ is monotonically complete (also known in the literature as having a Levi norm) if every norm-bounded upwards direct set $A\subset X_+$ has a supremum in $X$. Note that a monotonically complete Banach lattice is always Dedekind complete. Also note that dual Banach lattices, and KB-spaces (those Banach lattices where every increasing norm bounded sequence converges) are monotonically complete. For convenience, we will use the following terminology: Given $\lambda\geq1$, a Banach lattice $X$ has the $\lambda$-Fatou property if for every upwards directed subset $A \subset X_+$ with supremum $b$, we have $$\| b\| \leq \lambda \sup \{\|a\|:a\in A\}.$$ Similarly, we will say $X$ has the $\lambda$-Nakano property if for every upwards directed order bounded subset $A \subset X_+$, we have  $$\inf\{\| b\|:b\text{ an upper bound for }A\} \leq \lambda \sup \{\|a\|:a\in A\}.$$

\begin{proposition}\label{p:monotonicallycomplete}
Suppose $X$ is a monotonically complete Banach lattice. For every $\lambda\geq1$, the following are equivalent:
\begin{enumerate}
\item $X$ has the $\lambda$-Fatou property.
\item $X$ has the $\lambda$-Nakano property.
\item $X$ has the $\lambda$-strong Nakano property.
\end{enumerate}
\end{proposition}

\begin{proof} 
$(1)\Rightarrow (2)$
Suppose that the norm on $X$ has the $\lambda$-Fatou property. Let $A \subset X_+$ be an  upwards directed order bounded subset. As $X$ is monotonically complete, $A$ has a supremum $y_0$ in  $X$. Hence, we get $\|y_0 \| \leq \lambda \sup \{\|a\|:a\in A\}.$ Let $B$ denote the set of all upper bounds for $A$. It follows that
$$ 
\inf \{\|b\|:b\in B\} \le \|y_0 \|\leq \lambda\sup\{\|a\|,a\in A\}.
$$
As the converse inequality always holds, this finishes the proof.

$(2)\Rightarrow (3)$
Let $A$ be an upwards directed norm bounded subset of $X_+$. Since $X$ is monotonically complete and has the $\lambda$-Nakano property, it follows that $A$ has a supremum $y$ and  we have
$$
\|y\|\leq \inf \{\|b\|:b \in B \} \leq \lambda \sup \{\|a\|: a\in A\},
$$
 where $B$ is the set of all upper bounds for $A$.
 
$(3)\Rightarrow (1)$ always holds.
\end{proof}

As a consequence of Proposition \ref{p:monotonicallycomplete} and \cite[Proposition 2.4.19]{M-N} we have a fairly large class of Banach lattices with the strong Nakano property:

\begin{corollary}\label{c:dual}
Every dual Banach lattice has the strong Nakano property.
\end{corollary}

It is well known that KB-spaces are in particular order continuous \cite[Theorem 2.4.2]{M-N}. Within the class of Banach lattices with the strong Nakano property, the converse also holds:

\begin{proposition}
Suppose $X$ is a Banach lattice. The following are equivalent:
\begin{enumerate}
\item $X$ is a KB-space.
\item $X$ has order continuous norm and the strong Nakano property.
\end{enumerate}
\end{proposition}

\begin{proof}
$(1)\Rightarrow (2)$: Let $A$ be an upwards directed norm bounded set in $X_+$. As $X$ is a KB-space then by \cite[Theorem 2.4.2]{M-N} it is a projection band in $X^{**}$. By Corollary \ref{c:dual} there is $y$ in $X^{**}_+$ an upper bound of $A$ such that $\|y\|= \sup { \{\|a\|: a\in A}\}$. If $P$ denotes the band projection onto $X$, it follows that $Py$ is an upper bound of $A$ in $X$ with $\|Py\|\leq \|y\|$. Hence $X$ has the strong Nakano property. Order continuity follows from \cite[Theorem 2.4.2]{M-N}.

$(2)\Rightarrow (1)$:
 Let $ (x_n)_{n \in \mathbb{N}}$ be a monotone sequence in the unit ball of $X$. Without loss of generality, we may assume that $(x_n) _{n \in\mathbb{N}}$ is an increasing positive sequence. Since $X$ has the strong Nakano property there is $y_0$ in $X_+$ an upper bound of $ (x_n) _{n \in \mathbb{N}}$. It follows from order continuity and \cite[Theorem 2.4.2(iii)]{M-N} that $(x_n)_{n \in\mathbb{N} }$ is convergent. Hence, $X$ is a KB-space.
 \end{proof}

Looking at general directed sets for the strong Nakano property is not needed when the density character of the Banach lattice in question is small. In particular, we have the following.

\begin{lemma}\label{l:separable strong Nakano}
Let $X$ be a separable Banach lattice and $\lambda\geq1$. $X$ has the $\lambda$-strong Nakano property if and only if it has the  $\lambda$-strong $\sigma$-Nakano property.
\end{lemma}

\begin{proof}
Let us prove the non-trivial implication. Suppose that $X$ has the $\lambda$-strong $\sigma$-Nakano property.
Let $A$ be an upwards directed norm bounded subset in $X_+$. As $A$ is also separable then it contains a countable dense subset $A_0=\{a_1,a_2,\ldots\}$. Let us consider the sequence
$$y_n=\bigvee_{i=1}^n a_i,\quad\quad n\in\mathbb N.$$

Note $(y_n)_{n\in\mathbb N}$ is an increasing norm bounded sequence in $X_+$. By the $\lambda$-strong $\sigma$-Nakano property there is $u\in X_+$, an upper bound of $(y_n)_{n\in\mathbb N}$, hence of $A_0$, such that 
$$
\|u\|\leq \lambda\sup\{\|y_n\|:n\in\mathbb N\}\leq \lambda \sup\{\|a\|:a\in A\}.
$$

Now, for every $a \in A$ there is a sequence $(x_n)_{n \in \mathbb{N}}$ in $A_0$ that converges in norm to $a$. Hence, there is a subsequence $(x_{n_k})_{k \in \mathbb{N}} \subset A_0$ which converges in order to $a$. Since $x_{n_k} \leq u$ then $a\leq u$, so $u$ is also an upper bound of $A$ and we have $\|u\|\leq  \lambda \sup\{\|a\|:a\in A\},$ so $A$ has the $\lambda$-strong Nakano property.
\end{proof}

The previous equivalence does not remain true if the Banach lattice $X$ is not separable as the following simple example shows.

\begin{example}
Let $$X=\{x\in \ell_{\infty}(\mathbb{R}):supp(x)\;\textrm{is countable}\},$$ where $supp(x)$ denotes the support of $x\in \ell_\infty(\mathbb R)$. Given an increasing norm bounded sequence $(x_n)$ in $X_+$, let $M=\sup_n\|x_n\|$ and let $A=\bigcup_{n\in\mathbb N} supp(x_n)$. Let $x\in\ell_\infty(\mathbb R)$ take the value $M$ on $A$ and 0 elsewhere. Clearly, $x$ is an upper bound of $(x_n)_{n\in\mathbb N}$ in $X$ with $\sup\{\|x_n\|_{\infty}:n\in \mathbb{N}\}=M=\|x\|$, so $X$ has the strong  $\sigma$- Nakano property.

Now for each countable set $A\subset \mathbb R$ let $x_A =\chi_A$ denote the indicator function of $A$. Clearly, $(x_A)$ is upwards directed (for the inclusion) and norm bounded in $X_+$. However, if $x\in \ell_\infty(\mathbb R)$ is an upper bound of $(x_A)$ then $x$ must be at least 1 on every point of $\mathbb R$. Thus, $x\notin X$ and this space fails the strong Nakano property.
\end{example}

\subsection{Stability of the strong Nakano property}
Recall that if $X$, $Y$ are Banach lattices such that $Y$ is Dedekind complete, then the space of regular operators $L_r(X,Y)$ is a Dedekind complete Banach lattice with the regular norm (cf. \cite[Section 4.4]{AB}). We can characterize when this space has the strong Nakano property:

\begin{proposition}
Let $Y$ be a Dedekind complete Banach lattice. The following are equivalent:
\begin{enumerate}
    \item $Y$ has the strong Nakano property.
    \item For every Banach lattice $X$, the space of regular operators $L_r(X,Y)$ has the strong Nakano property.
    \item For some Banach lattice $X\neq\{0\}$, the space of regular operators $L_r(X,Y)$ has the strong Nakano property.
\end{enumerate}
\end{proposition}

\begin{proof}
$(1)\Rightarrow (2)$: Suppose that $Y$ has the strong Nakano property and let $X$ be any Banach lattice. Let $(T_\alpha)_{\alpha \in A}$ be an upwards directed norm bounded set in $L_r(X,Y)_+$. For every $x\in X_+$, $(T_\alpha x)_{\alpha}$ is an upwards directed norm bounded set in $Y_+$. Hence, by the strong Nakano property of $Y$, $(T_\alpha x)_{\alpha}$ has an upper bound in $Y$. By Dedekind completeness, the supremum of $(T_\alpha x)_{\alpha}$ exists in $Y_+$. Let us call $T_x$ this supremum and note that
$\|T_x\|=\sup\limits_{\alpha}\|T_{\alpha} x\|,$ for every $x\in X_+$.

Now, by \cite[Theorem 1.19]{AB}, the set $(T_\alpha)_{\alpha}$ has a supremum $T$ in $L_r(X,Y)$ and 
$$
T(x)=(\sup\limits_{\alpha}T_{\alpha})(x)=\sup\{T_{\alpha}(x):\alpha \in A\}=T_x.
$$ 
Moreover, 
$$
\|T\|=\sup_{x\in X_+,\|x\|\leq1}\|Tx\|=\sup_{x\in X_+,\|x\|\leq1}\|T_x\|=\sup_{x\in X_+,\|x\|\leq1}\sup_\alpha \|T_\alpha x\|=\sup_\alpha\|T_\alpha\|.
$$
Hence, $L_r(X,Y)$ has the strong Nakano property.

The implication $(2)\Rightarrow (3)$ is trivial.

$(3)\Rightarrow (1)$: Suppose that for some $X\neq\{0\}$, the space $L_r(X,Y)$ has the strong Nakano property. Let $(y_\alpha)_{\alpha}$ be an upwards directed norm bounded set in $Y_+$. Let $x_0 \in X_+$ with $\|x_0\|=1$. By Hahn-Banach theorem we can take $x_0^*\in X^*_+$ such that $x_0^*(x_0)=1$ and $\|x_0^*\|=1$. For each $\alpha$ let $T_{\alpha}$ be the rank-one operator
$$
T_{\alpha}(x)=x_0^*(x)y_\alpha,\quad\quad \text{for }x \in X.
$$
Clearly, $(T_{\alpha})_{\alpha}$ is an upwards directed norm bounded set in $L_r(X,Y)_+$, so by the strong Nakano property of $L_r(X,Y) $ and \cite[Theorem 1.19]{AB} there is $T$ in $L_r(X,Y)_+$ such that
$$T(x)=\sup\limits_{\alpha} T_{\alpha}(x),\quad\quad \text{for } x\in X_+,$$ with 
$\|T\|=\sup\limits_{\alpha} \|T_{\alpha}\|.$
It follows then that
$$\sup\limits_{\alpha} y_\alpha=\sup\limits_{\alpha} T_{\alpha}(x_0)=T(x_0)$$ and $$\|T(x_0)\|\leq \|T\|=\sup\limits_{\alpha} \|T_{\alpha}\|\leq \sup\limits_{\alpha} \|y_{\alpha}\|.$$
Hence, $Y$ has the strong Nakano property.
\end{proof}

\begin{lemma}
Let $X$ be a Banach lattice and let $Y$ be an order dense sublattice with the strong Nakano property. Then $X$ has the strong Nakano property.
\end{lemma}

\begin{proof}
Assume that $Y$ has the strong Nakano property. Let $(x_{\alpha})_{\alpha}$ be an upwards directed norm bounded set in $X_+$.
Let us consider the set
$$
A=\{y \in Y :0 <y \le x_{\alpha} \; \text{for some}\; \alpha  \}.
$$
Clearly, $A$ is an upwards directed (it is closed under finite suprema) norm bounded subset in $Y_+$. Hence, there is $b\in Y_+$ which is an upper bound of $A$ with
$\|b\|=\sup_{y \in A} \|y\|$. Since $Y$ is order dense the conclusion follows.
\end{proof}

\begin{proposition}\label{p:regular}
Every regular sublattice of a Banach lattice with the Fatou property also has the Fatou property.
\end{proposition} 

\begin{proof}
Suppose that $X$ has the Fatou property. Let $Y$ be a regular sublattice of $X$ and let $A$ be an upwards directed order bounded subset of $Y_+$  with supremum $b$.
$A$ is also  an upwards directed order bounded subset of $X_+$ and  $b$ is also the  supremum  of $A$ in $X_+$ ($Y$ is a regular sublattice of $X$) so using that $X$ has the Fatou property we have
$\|b\| = \sup \{\|a\|:a\in A\}.$
\end{proof}

\begin{remark}
The previous proposition is not true if we replace the Fatou property by the strong Nakano property. For example, one can consider $c$ the space of convergent sequences which has the strong Nakano property (as it is a $C(K)$ space). However, $c_0$ is a regular sublattice (even an ideal) of $c$ that fails the strong Nakano property.
\end{remark}

\subsection{Projections and the strong Nakano property}
\begin{proposition}\label{p:stability of a complemented sublattice}
Suppose $X$ is a Banach lattice with the strong Nakano property and $Y\subset X$ is a closed sublattice complemented by a regular projection $P$, then $Y$  has also the $\|P\|_r$-strong Nakano property.
\end{proposition}

\begin{proof}
Let $Y$ be a sublattice of $X$ which is complemented by a regular projection $P:X\rightarrow Y$ and let $A$ be an upwards directed norm bounded subset of $Y_+$.
By the strong Nakano property of $X$, there is $b$ in $X_+$ which is an upper bound of $A$ such that $\|b\| = \sup  \{\|a\|:a\in A\}$.
Let $c=|P|(b)\in Y$. Note that for every $a\in A$ we have
$$
a=Pa\leq|P|a\leq|P|b=c,
$$
hence, $c$ is also an upper  bound of $A$. Clearly, we have
$$
\|c\|=\||P|b\| \le \|P\|_r \sup \{\|a\|:a\in A\}.
$$
\end{proof}

As a consequence, every projection band in a Banach lattice with the strong Nakano property, must also have the strong Nakano property. However, this is no longer true for arbitrary bands: Consider in $C[0,1]$ the band $$X=\{f\in C[0,1]: f|_{[0,\frac12]}=0\}.$$ For $n\in\mathbb N$ consider the function $f_n$ such that $f_n|_{[0,\frac12]}=0$, $f_n|_{[\frac12+\frac1n,1]}=1$ and $f_n$ affine on $[\frac12,\frac12+\frac1n]$. Clearly, the sequence $(f_n)_{n\in\mathbb N}\subset X$ is increasing and norm bounded but has no upper bound in $X$.

Recall that a Banach lattice $Z$ is projective if for every Banach lattice $X$, $J\subset X$ a closed ideal, every lattice homomorphism $T:Z\rightarrow X/J$ and every $\varepsilon>0$, there exists a lattice homomorphism $\hat T:Z\rightarrow X$ such that $T=\hat T\circ Q$ (where $Q:X\rightarrow X/J$ denotes the quotient map) and $\|\hat T\|\leq (1+\varepsilon)\|T\|$.

\begin{corollary}
Every projective Banach lattice has the $1_+$-strong Nakano property.
\end{corollary}

\begin{proof}
By \cite[Theorem 10.3]{dePW}, for every $\varepsilon>0$ every projective Banach lattice is $(1+\varepsilon)$-isomorphic to a sublattice of $FBL(A)$, the free banach lattice generated by some set $A$, and there is a lattice homomorphism projection $P:FBL(A)\rightarrow FBL(A)$ onto this sublattice with $\|P\|\leq 1+\varepsilon$.
We know that $FBL(A)=FBL[\ell_1(A)]$  (\cite[Corollary 2.9.iii]{ART}) and $FBL[\ell_1(A)]$ has the strong Nakano proprety for every set $A$ (\cite[Theorem 4.11 ]{ART}). The conclusion follows from Proposition \ref{p:stability of a complemented sublattice}.
\end{proof}

Note that the converse of this last result is not true in general, as there exist $C(K)$ spaces which are not projective (cf. \cite{dePW}).

\begin{proposition}
If $Y$ is a Banach lattice with the strong Nakano property and $Y$ is an ideal in $X$, then it is a projection band.
\end{proposition}

\begin{proof}
    \cite[Proposition 12.3]{dePW}.
\end{proof}

\begin{proposition}
Suppose for every $i\in I$, $E_i$ is a Banach lattice with the $\lambda_i$-strong Nakano property. Then the direct sum $\ell_{p}(\prod_{i\in I} E_i)$ has the ($\sup_{i\in I}\lambda_i$)-strong Nakano property for $1\le p \leq\infty$.
\end{proposition}

\begin{proof}
For every $k\in I$, let
$$
\begin{array}{cccc}
\pi_k&:\ell_p(\prod_{i\in I} E_i)& \longrightarrow &E_k\\
 &(x_i)_{i \in I} & \longmapsto & x_k
\end{array}
$$
be the canonical projection onto $E_k$, which is a lattice homomorphism. Let $(x_{\alpha})_{\alpha \in A}$ be an upwards directed norm bounded set in $\ell_{p}(\prod_{i\in I} E_i)_+$. For every $k\in I$, $(\pi_k(x_{\alpha}))_{\alpha \in A}$ is an upwards directed norm bounded set in $(E_k)_+$. Hence, by the $\lambda_k$-strong Nakano, there is $b_k\in E_k$ an upper bound of $(\pi_k(x_{\alpha}))_{\alpha \in A}$ with
$$
\|b_k\|_{E_k}\leq \lambda_k\sup_{\alpha\in A}\|\pi_k(x_{\alpha})\|_{E_k}. 
$$
Clearly, $b=(b_i)_{i\in I}$ is an upper bound of $(x_{\alpha})_{\alpha \in A}$ and
\begin{eqnarray*}
    \|b\|&=&\Big(\sum_{i\in I}\|b_i\|^p\Big)^{\frac1p}\leq \Big(\sum_{i\in I}\lambda_i^p\sup_{\alpha\in A}\|\pi_i(x_{\alpha})\|^p\Big)^{\frac1p}\\
    &\leq& \sup_{i\in I}\lambda_i\Big(\sum_{i\in I}\sup_{\alpha\in A}\|\pi_i(x_{\alpha})\|^p\Big)^{\frac1p}=\sup_{i\in I}\lambda_i\sup_{\alpha\in A}\Big(\sum_{i\in I}\|\pi_i(x_{\alpha})\|^p\Big)^{\frac1p}\\
    &=&\sup_{i\in I}\lambda_i\sup_{\alpha\in A}\|x_\alpha\|
\end{eqnarray*}
where the identity on the second line follows from the KB-property of $\ell_p(I)$ (with obvious changes for the case $p=\infty$).
\end{proof}

\section{Free Banach lattices and the strong Nakano property}\label{s:fbl}

It was proved in \cite{ART} that $FBL[\ell_1(A)]$ has the strong Nakano property for every index set $A$. Our motivation for the results in this section is to understand what conditions on a Banach space $E$ guarantee that $FBL[E]$ has the strong Nakano property. Similarly, suppose $FBL[E]$ has the strong Nakano property, what can we say about the underlying $E$?

We will start looking at the simpler case of free Banach lattices generated by finite dimensional spaces. Recall in this case that $FBL[E]$ can be identified with an appropriate renorming of $C(S_{E^*})$ the space of continuous functions on the unit sphere of $E^*$ (see \cite{OTTT}). However, the constant in this renorming grows without bound when the dimension of $E$ increases. This easily implies that, as long as $E$ is finite dimensional, $FBL[E]$ has the $\lambda_E$-strong Nakano property for an appropriate constant $\lambda_E$. We will show next that actually there is a universal constant $\lambda$ valid for all finite dimensional $E$.

Let us introduce the following notation.

\begin{definition}
Given $f$ in $H_1[E]$ and $k \in \mathbb{N}$, let $$\|f\|_{FBL_k[E]} = \sup\Big\{\sum\limits_{i=1}^{k}|f(x_i^*)|:\, \, x_1^*,\dots,x_{k}^*\in E^*, \,  \sup\limits_{x\in B_E} \sum\limits_{i=1}^{k} |x_i^\ast(x)|\leq 1 \Big\}.
$$

\end{definition}

Obviously, for every $f\in H_1[E]$, $(\|f\|_{FBL_k[E]})_{k\in\mathbb N}$ is a non-decreasing sequence and we have
$$
\|f\|_{FBL[E]}=\sup\limits_{k \in \mathbb{N}}\|f\|_{FBL_k[E]}.
$$
    
The following is a consequence of the work of \cite{SZ} for estimating the summing norm of finite dimensional operators with  few vectors:

\begin{lemma}\label{l:majoration of the free norm}
 There exist a universal constant $\lambda_{fin}$ and a sequence of integers $(k_m)_{m \in \mathbb{N}}$ such that for every finite dimensional Banach space $E$  we have:
$$\|f\|_{FBL[E]}\le \lambda_{fin} \|f\|_{FBL_{k_{dim(E)}}[E]},
$$
for every $f$ in $H_1[E].$ 
\end{lemma}

\begin{proof}
Let $f\in H_1[E]$, $n=dim(E)$ and pick $x_1^*,\dots,x_{N}^*\in E^*$ (with arbitrary $N \in \mathbb{N}$) such that
$$\sup_{x\in B_E} \sum_{k=1}^{N} |x_k^\ast(x)|\leq 1.$$
Set $a=\sum_{k=1}^{N} |f(x_k^*)|.$
By \cite[Lemma 6]{SZ} there exists $k_n$ a positive integer which depends only on $n$, so that the statement of \cite[Lemma 6]{SZ} applied to $E^*$ with $w_k=|f(\frac{x_k^*}{a})|$, yields $\sigma \subset \{1,\dots,N\}$ and scalars $(\mu_k)_{k \in \sigma }$ such that:
\begin{enumerate}
\item $|\sigma|\le k_n$,
\item $\sup_{x\in B_E} \sum_{k \in \sigma} |\mu_k x_k^*(x)|\le 1/C,$
\item $\sum_{k \in \sigma} w_k \mu_k \ge C,$
\end{enumerate}
where $C$ is a universal constant.

It follows that
\begin{align*}
  \ C^2&\le \sum_{k \in \sigma} C w_k \mu_k = \sum_{k \in \sigma} C|f(\frac{x_k^*}{a})| \mu_k\\
  &\le \frac1a\sum_{k \in \sigma} |f(C x_k^*)| |\mu_k|
  =\frac1a\sum_{k \in \sigma} |f(C x_k^* |\mu_k|)|
  \le \frac1a \|f\|_{FBL_{k_n}[E]}.
\end{align*}
Thus,
$\sum_{k=1}^{N} |f(x_k^*)| C^2\le  \|f\|_{FBL_{k_n}[E]}$. Since $(x_k^*)_{k=1}^N$ are arbitrary, setting $\lambda_{fin}=C^{-2}$ we get
$$\|f\|_{FBL[E]}\le \lambda_{fin} \|f\|_{FBL_{k_n}[E]}.$$
\end{proof}

\begin{lemma}\label{l:existence of upper bound}
Let $E$ be a Banach space of finite dimension. For every $k\in\mathbb N$, every $\varepsilon>0$ and every $f\in H_1[E]_+$ there is $g\in FBL[E]$ such that
\begin{enumerate}
    \item $f\le g$,
    \item $\|g\|_{FBL_{k}[E]}\le (1+\varepsilon) \|f\|_{FBL_{k}[E]}.$
\end{enumerate}
\end{lemma}

\begin{proof}
Let $\delta> 0$ small enough to be chosen later. As $E$ has finite dimension, $S_{E^*}$ is compact, so there exist $y_1^*,\ldots,y_p^*$ in $S_{E^*}$ such that
$$
S_{E^{\ast }}=\bigcup\limits_{i=1}^{p}B\left(
y_{i}^{\ast },\delta\right),
$$
where $B\left(
y_{i}^{\ast },\delta\right)=\{x^* \in S_{E^*}:\,\|y_{i}^{\ast }-x^{\ast }\|\le \delta \}$.

By Tietze's Theorem, for every $1\leq i\leq p$ we can find $g_{i}\in C\left( S_{E^{\ast }}\right) $
satisfying $g_{i}=0$ on $S_{E^{\ast }}\backslash B\left( y_i^{\ast
}, 2\delta \right)$,  $g_{i}^{\ast }=\sup\limits_{x^{\ast }\in B\left( y_i^{\ast },2\delta\right)}f\left( x^{\ast }\right)$ on $B\left( y_i^{\ast },\delta\right)$, and with values between these two elsewhere.

Let $g=\bigvee\limits_{i=1}^{p}g_{i}$. Clearly, we have $g\in FBL[E]$ and $0\le f \le g$. Indeed, for $x^*$ on $S_{E^*}$ there is $i$ such that $x^* \in B\left( y_i^{\ast },\delta\right)$ and so
$$
f(x^*) \le \sup\limits_{y^*\in B\left( y_i^{\ast },2\delta\right)}f\left( y^{\ast }\right)= g_i(x^*) \le g(x^*).
$$

Now, let $x^*$ in $S_{E^*}$. By definition, there is $i\in \{1,\ldots,p\}$ such that $g(x^*)=g_{i}(x^*)$. We want to show that there exists $v^*$  such that
\begin{enumerate}
    \item $g(x^*)\le f(v^*) + \delta$,
    \item $\|x^*-v^*\| \le 4\delta$.
\end{enumerate}
Indeed, we have $g(x^*)=g_{i}(x^*)\le \sup\limits_{y^*\in B\left( y_i^{\ast },2\delta\right)}f\left( y^{\ast }\right)$, so one can find $v^*$ in $B\left( y_i^{\ast
},2\delta \right)$ with $g(x^*)=g_{i}(x^*)\le f(v^*) + \delta$.
As the case $g(x^*)=0$ is trivial, we can assume that $g(x^*)=g_i(x^*)>0$, which implies $x^*$ is in $B\left( y_i^{\ast},2\delta \right)$, so we get
$\|x^*-v^*\| \le 4\delta.$

Now, we claim that for $k\in\mathbb N$, $\varepsilon>0$, choosing $\delta > 0$ small enough we have that $$
\|g\|_{FBL_{k}[E]} \le (1+\epsilon) \|f\|_{FBL_{k}[E]}.
$$
To check this, fix $x_{1}^{\ast },\ldots,x_{k}^{\ast }$ in $E^{\ast }$ satisfying $\sup\limits_{x\in B_{E}}\sum\limits_{i=1}^{k}\left\vert x_{i}^{\ast }\left(
x\right) \right\vert \leq 1.$ For each $1\leq i\leq k$, choose as above  $v_{i}^{\ast }$ such that $\|\dfrac{x_{i}^{\ast }}{\left\Vert x_{i}^{\ast
}\right\Vert }-v_i^*\| \le 4\delta$ and $g(\dfrac{x_{i}^{\ast }}{\left\Vert x_{i}^{\ast
}\right\Vert })\le f(v_i^*) + \delta$. Let $t_{i}^{\ast }=\left\Vert x_{i}^{\ast }\right\Vert v^*_{i}.
$
 
We have that 
\begin{eqnarray*}
    \sum\limits_{i=1}^{k}g\left( x_{i}^{\ast }\right)
 &=&\sum\limits_{i=1}^{k}\left\Vert x_{i}^{\ast }\right\Vert
 g\left( \dfrac{x_{i}^{\ast }}{\left\Vert x_{i}^{\ast }\right\Vert 
}\right)\\
&\leq& \sum\limits_{i=1}^{k}\left\Vert x_{i}^{\ast }\right\Vert \left(
 f\left( v_{i}^{\ast }\right)  +\delta \right)
=\sum\limits_{i=1}^{k}\left( f\left( t_{i}^{\ast }\right)
 +\left\Vert x_{i}^{\ast }\right\Vert \delta \right) \\
&\leq& \sum\limits_{i=1}^{k}  f\left( t_{i}^{\ast }\right)+k\delta \leq \left\Vert f\right\Vert
_{FBL_{k}[E]}\sup\limits_{x\in B_{E}}\sum\limits_{i=1}^{k}\left\vert
t_{i}^{\ast }\left( x\right) \right\vert +k\delta \\
&\leq&  \|f\|_{FBL_{k}[E]} (4k\delta+1)+k\delta.
 \end{eqnarray*}

For $\epsilon >0$ and $k\in \mathbb N$ we can choose $\delta$ such that $\|g\|_{FBL_{k}[E]} \le (1+\epsilon) \|f\|_{FBL_{k}[E]}.$
\end{proof}

\begin{theorem}
For every finite dimensional Banach space $E$, $FBL[E]$ has the $(\lambda_{fin})_+$-strong Nakano property (where $\lambda_{fin}$ is the universal constant coming from Lemma \ref{l:majoration of the free norm}).
\end{theorem}

\begin{proof}
By Lemma \ref{l:separable strong Nakano} it is enough to check the strong Nakano property for sequences. Let $(f_n)_{n\in\mathbb{N}}$ be an increasing norm bounded sequence in $FBL[E]_+$ with $\sup_{n\in\mathbb{N}} \|f_n\|_{FBL[E]}=1$. Define $f$ on $E^*$ by $f(x^*):=\sup_{n\in\mathbb{N}}\{f_n(x^*)\}$ for every $x^*$ in $E^*$. By \cite[Lemma 4.2]{ART} we have $f\in H_1[E]$ with $\|f\|_{FBL[E]}=1$.

By Lemma \ref{l:existence of upper bound} there is $g$ is in $FBL[E]$ such that $f\le g$ and $\|g\|_{FBL_{k_{dim(E)}}[E]}\le (1+\varepsilon) \|f\|_{FBL_{k_{dim(E)}}[E]}$.

Therefore, $g$ is an upper bound of  $(f_n)_{n\in\mathbb{N}}$ and by Lemma \ref{l:majoration of the free norm} we have
\begin{align*}
\|g\|_{FBL[E]}&\le \lambda_{fin} \|g\|_{FBL_{k_{dim(E)}}[E]}\\
&\le (1+\epsilon)\lambda_{fin} \|f\|_{FBL[E]}=(1+\epsilon)\lambda_{fin}.
\end{align*}
This completes the proof.
\end{proof}

The rest of this section will be devoted to show some conditions that a Banach space $E$ must satisfy in case $FBL[E]$ has the strong Nakano property.

Let us first recall that a subspace $F$ of a Banach space $E$ is called an ideal (cf.~\cite{GKS}) if the subspace 
$$
F^\perp=\{x^*\in E^*:x^*(y)=0 \text{ for }y\in F\}
$$ 
is the kernel of a contractive projection on $E^*$. This notion should not be confused with that of an ideal in a Banach lattice.

\begin{proposition}
Suppose $E$ is a Banach space such that $FBL[E]$ has the $\lambda$-strong Nakano property for some $\lambda$. Then $E$ cannot contain a subspace isomorphic to $c_0$.
\end{proposition}

\begin{proof}
Suppose that $j:c_0\rightarrow E$ is an isomorphic embedding. By \cite[Corollary 4.12]{OTTT} the extension $\overline{j}:FBL[c_0]\rightarrow FBL[E]$ is a lattice isomorphic embedding. By \cite[Theorem 4.1]{AMRT}, there is a lattice embedding $\alpha:c_0\rightarrow FBL[c_0]$ such that $\beta\alpha=id_{c_0}$, where $\beta=\widehat{id_{c_0}}:FBL[c_0]\rightarrow c_0$ is the unique lattice homomorphism extending the identity on $c_0$, $id_{c_0}$.

Let $s_n$ denote the summing basis in $c_0$, and $f_n=\overline{j}\alpha s_n\in FBL[E]$. Clearly, $(f_n)_{n\in\mathbb N}$ is an increasing sequence with $\sup_{n\in\mathbb N} \|f_n\|\leq\|\alpha\|\|j\|$. Since $FBL[E]$ has the $\lambda$-strong Nakano property, there is $f\in FBL[E]$ such that $f_n\leq f$ for every $n\in\mathbb N$ and $\|f\|\leq \lambda \|\alpha\|\|j\|$.

Using the density of $FVL[E]=lat(\delta_E(E))$ in $FBL[E]$, we can find a separable subspace $F\subset E$ such that $f\in FBL[F]$. Enlarging $F$ if needed we can assume without loss of generality that $j(c_0)\subset F$ and that $F$ is an ideal in $E$ (the latter due to \cite{SY}). In particular, by \cite[Corollary 4.14]{OTTT} we have $\|f\|_{FBL[F]}=\|f\|_{FBL[E]}$.

By Sobczyk's theorem there is a projection $P:F\rightarrow j(c_0)$. Let us consider $x=\beta\overline{j^{-1}}\overline{P}f\in c_0$. Observe that for every $n\in\mathbb N$ we have
$$
s_n=\beta\alpha s_n=\beta\overline{j^{-1} P j}\alpha s_n=\beta \overline{j^{-1}}\overline{P}f_n\leq \beta \overline{j^{-1}}\overline{P}f=x. 
$$
This is a contradiction as $x$ would be larger than the constant 1 sequence, hence $x\notin c_0$.
\end{proof}

It was also shown in \cite[Theorem 4.13]{ART} that $FBL[L_1[0,1]]$ fails the Fatou property, hence strong Nakano. We can improve this further as follows:

\begin{theorem}
If $FBL[E]$ has the $\lambda$-Fatou property for some $\lambda$, then $E$ does not contain a closed subspace isomorphic to $L_1[0,1]$.
\end{theorem}

\begin{proof}
For convenience let $L_1=L_1[0,1]$. Let us first assume $E$ is separable. Suppose that $E$ contains a closed subspace isomorphic to $L_1$. Hence, let $T:L_1 \rightarrow E$ be an operator  such that $T(L_1)=F$ is isomorphic to $L_1$. Without loss of generality we can assume that $\|T\|=1$.

For each $n \in \mathbb{N}$, define  
$I_{n,j}:=[\frac{j-1}{2^n},\frac{j}{2^n}]$ for all $ j \in \{1,...,2^n\}$ and set 
$y_{n,j} :=T(\chi_{I_{n,j}})$.

We define $f_n =\sum\limits_{i=1}^{2^n}  |\delta_{y_{n,j}}|\in FBL[E]_+$ which satisfy
$$
\|f_n\|_{FBL[E]} \le  \sum\limits_{i=1}^{2^n}  \|\delta_{y_{n,j}}\| =\sum\limits_{i=1}^{2^n}  \|y_{n,j}\| \le   \sum\limits_{i=1}^{2^n}  \|\chi_{I_{n,j}}\|_{L_1} =1.
$$

It is easy to check (see \cite[Lemma 4.12]{ART}) that $f_n\leq f_{n+1}$ and for every $y^*$ in $E^*$ 
$$
\lim_{n \to \infty} f_n(y^*)=\|T^*y^*\|_{L_1}.
$$

Now, fix any $g \in FBL[E]_+ $ with $ \|g\| _{FBL[E]} > 1 $  and let $\Tilde{f_n}:= g \wedge f_n$. The sequence $\Tilde{f_n}$ is increasing , bounded above by $g$
and satisfies $$\|\Tilde{f_n}\|_ {FBL[E]}  \le \|f_n\|_{FBL[E]} \le 1.$$
We will show that 
$ \sup_{n \in \mathbb{N}} \Tilde{f_n}  = g $ in $FBL[E].$ 

To this end, let $(r_j)_{j\in \mathbb{N}}$ denote the sequence of Rademacher functions, which converges to 0 in the $w^*$-topology of $L_1^*=L_{\infty}$. Thus, we can define  $S_0:L_1 \rightarrow c_0$ by 
$$
S_0 (f) =\bigg(\int r_n fd\mu\bigg)_{n\in \mathbb{N}}.
$$ 
Note $S_0 \circ T^{-1}$ is an operator from $F\subset E$ to $ c_0$ and, as $E$ is separable, by Sobczyk's theorem there is an extension $S:E \rightarrow c_0$ with $S|_F=S_0 T^{-1}$. Hence, we can consider a sequence $(\tilde {r}_j)_{j\in \mathbb{N}}\subset E^*$ such that $S(f) =(\langle \tilde{r}_n,f\rangle)_{n\in \mathbb{N}}$ for every $f$ in  $E$. In other words, $(\tilde {r}_j)_{j\in \mathbb{N}}$ is $w^*$-null in $E^*$.

Given $h$ in $ E^*$, let $K =\|g\|_{FBL[E]}  \|h\|  +  \|T^*h\|_{L_1} +1$, and define $h_j =h + K \tilde{r}_j$. The sequence $(h_j)$ is $w^*$-convergent to $h$ in $E^*$. Since $g$ is $w^*$- continuous on bounded sets, it follows that $g(h_j)\rightarrow g(h)$ as $j\rightarrow \infty$. Therefore, we can find $j_0 \in \mathbb{N}$ such that for every $j\ge j_0$ we have
\begin{equation}\label{eq:K}
g(h_j) \le g(h) +1 \le \|g\|_{FBL[E]}  \|h\|+1 \le K-\|T^*h\|_{L_1}.
\end{equation}

We claim that for $j\geq j_0$, $g(h_j)\le \|T^*h_j\|_{L_1}$. Indeed, we have
$$
\|T^* h_j\|_{L_1} = \|T^* h + KT^*\tilde{r_j}\|_{L_1} \geq K \|T^* \tilde{r_j}\|_{L_1} -\|T^*h\|_{L_1} $$ 
and by construction of $\tilde{r_j}$ we have
$$
\|T^* \tilde{r_j}\|_{L_1} \geq \langle T^*\tilde{r_j},r_j\rangle =\langle \tilde{r_j},Tr_j\rangle =\langle r_j,T^{-1}(Tr_j)\rangle=\langle r_j,r_j\rangle=1.
$$
Hence, by \eqref{eq:K} we get
$$
\|T^* h_j\|_{L_1} \geq K-\|T^*h\|_{L_1} \geq  g(h_j),
$$ as claimed.
Now let $\phi \in FBL[E] $ be any upper bound of  $(\Tilde{f_n})_{n \in \mathbb{N}}$. For $j\ge j_0$, we have  
$$
\phi(h_j) \ge sup_{n \in \mathbb{N}} \Tilde{f_n} (h_j)= g(h_j) \wedge sup_{n \in \mathbb{N}} f_n(h_j)=g(h_j) \wedge \|T^*h_j\|_{L_1} = g(h_j).
$$
Thus, by the $w^*$ continuity of $g$ and $\phi$ we have $g (h)\le \phi(h)$ for every $h\in E^*$. Therefore, $g=\sup_{n \in \mathbb{N}} \Tilde{f_n} $ and we get a contradiction because  $\|g\|_{FBL[E]} > 1$ was arbitrarily large but $ \|\Tilde{f_n}\|_{FBL[E]} \le 1$.

To finish the proof, let us assume now that $E$ is non-separable and contains a closed subspace isomorphic to $L_1$. By \cite{SY}, there is $F\subset E$ a separable ideal containing a subspace isomorphic to $L_1$. Note that by \cite{OTTT} $FBL[F]$ is a regular sublattice of $FBL[E]$, so if $FBL[E]$ had the $\lambda$-Fatou property, using Proposition \ref{p:regular}, $FBL[F]$ also would have the $\lambda$-Fatou property. This is a contradiction with the first part of the proof.
\end{proof}

\section{Strong Nakano of $FBL^{(p)}[\ell_1(A)]$}\label{s:fblpell1}

The purpose of this section is to show that the proof of \cite[Theorem 4.11]{ART} that $FBL[\ell_1(A)]$ has the strong Nakano property can actually be adapted for the free $p$-convex Banach lattice $FBL^{(p)}[\ell_1(A)]$ with minor adjustments. We have decided to include details here as it will be convenient for future reference.

As in \cite{JLTTT}, given a Banach space $E$, let $H_p[E]$ denote the set of positively homogeneous functions on $E^*$ equipped with the norm
\begin{equation}\label{eq: norm fbp}
	\|f\|_{FBL^{(p)}[E]}=\sup\Big\{\Big(\sum_{i=1}^m |f(x_i^*)|^p\Big)^{\frac1p} :m\in \mathbb N, (x_i^*)_{i=1}^m\subset E^*, \sup_{x\in B_E} \sum_{i=1}^m |x_i^*(x)|^p\leq 1\Big\},
\end{equation}
and the pointwise order and lattice operations. The free $p$-convex Banach lattice generated by $E$ can be identified with the closed sublattice of $H_p[E]$ generated by the evaluation functionals $\delta_x:E^*\rightarrow \mathbb R$ given by $\delta_x(x^*)=x^*(x)$ for $x\in E$, $x^*\in E^*$. This is in the sense that for every $p$-convex Banach lattice $X$ and every operator $T:E\rightarrow X$ there is a unique lattice homomorphisms $\hat T:FBL^{(p)}[E]\rightarrow X$ extending $T$ (see \cite{JLTTT} for details). 

It should be noted that in the case $p=\infty$, $FBL^{(\infty)}[E]$ is always an AM-space, so by Proposition \ref{p:AM} and \cite[Proposition 9.1]{OTTT} we get that $FBL^{(\infty)}[E]$ has the Strong Nakano property precisely when $E$ is finite dimensional.

\begin{lemma}\label{lem:PointwiseSupremum}
Let $E$ be a Banach space and let $A$ be an upwards directed and pointwise bounded set in $H_p[E]$. Define $g: E^*\to \mathbb{R}_+$ 
by 
$$
g(x^*):=\sup\{f(x^*):f\in A \}\quad\quad \text{for  } x^*\in E^*.
$$
It follows that $g\in H_p[E]_+$ with 
$$
	\|g\|_{{FBL}^{(p)}[E]} = \sup \{\|f\|_{{FBL}^{(p)}[E]} : \, f\in A\}.
$$ 
\end{lemma}
\begin{proof}
Clearly, $g$ is positively homogeneous and $\|g\|_ {{FBL}^{(p)}[E]} \geq \|f\|_{{FBL}^{(p)}[E]}$ for every $f\in A$. 
To prove the converse inequality $\|g\|_{{FBL}^{(p)}[E]} \leq \sup \{\|f\|_{{FBL}^{(p)}[E]} : \, f\in A\}:=\alpha$ we can assume without loss of generality that
the supremum is finite. Fix $\epsilon>0$ and take any $x_1^*,\dots,x_n^* \in E^*$ such that $\sum_{k=1}^n |x_k^*(x)|^p\leq 1$
for every $x\in B_E$. Since $A$ is upwards directed, we can find $f\in  A$ such that
$g(x_k^*)^p-\frac{\epsilon}{n} \leq f(x_k^*)^p$ for all $k\in \{1\dots,n\}$, therefore
$$
	\sum_{k=1}^n g(x^*_k)^p \leq \epsilon+\sum_{k=1}^n f(x^*_k)^p \leq \epsilon+\alpha^p.
$$
It follows that $\|g\|_ {{FBL}^{(p)}[E]}^p\leq \epsilon+\alpha^p$. As $\epsilon>0$ is arbitrary, $\|g\|_{{FBL}^{(p)}[E]} \leq \alpha$.
\end{proof}

\begin{definition}
Let $E$ be a Banach space. We say that $f\in H_p[E]_+$ is {\em maximal} if 
$$
	\big\{g\in H_p[E]_+: \, g \geq f, \, \|g\|_{{FBL}^{(p)}[E]} = \|f\|_{{FBL}^{(p)}[E]} \big\}=\{f\}.
$$
\end{definition}

A standard application of Zorn's lemma allows us to prove the following (see \cite[Lemma 4.4]{ART}):

\begin{lemma}\label{lem:Zorn}
Let $E$ be a Banach space and $f\in H_p[E]_+$. Then there exists a maximal $\tilde{f} \in H_p[E]_+$ such that 
$f\leq \tilde{f}$  and $\|f\|_{{FBL}^{(p)}[E]} = \|\tilde{f}\|_{{FBL}^{(p)}[E]}$.
\end{lemma}

\begin{lemma}\label{lem:PropertiesMaximal}
Let $A$ be a non-empty set and let $f\in H_p[\ell_1(A)]_+$ be maximal. 
\begin{enumerate}
\item[(i)] $f(x^*)\leq f(y^*)$ whenever $x^*,y^*\in \ell_{\infty}$ satisfy $|x^*|\leq |y^*|$.
\item[(ii)]  $\big(\sum\limits_{k=1}^n (f(x^*_k))^p\big)^{\frac{1}{p}} \leq f\big((\sum\limits_{k=1}^n (x^*_k)^p\big)^{\frac{1}{p}})$
for every $n\in\mathbb N$ and $x_1^*,\dots,x_n^*\in (\ell_{\infty})_+$.
\item[(iii)] $\|f\|_{{FBL}^{(p)}[E]} = \|f\|_\infty$.
\end{enumerate}
\end{lemma}
\begin{proof} For any $z^*\in \ell_{\infty}$ let $R(z^*):=\{\lambda z^*:  \lambda>0\}\subset \ell_{\infty}$.

(i): We argue by contradiction. Suppose that $f(x^*)>f(y^*)$ and define $g:\ell_{\infty}\to \mathbb{R}_+$ by   
$$
	\begin{cases}
		g(\lambda y^*):=f(\lambda x^*) & \text{for all $\lambda>0$}, \\
		g(z^*):=f(z^*) & \text{for all $z^*\in \ell_{\infty} \setminus R(y^*)$}.
	\end{cases}
$$
It is easy to check that $g$ is positively homogeneous. Since $f(x^*)>f(y^*)$, we have $f\leq g$ and $f\neq g$.
Bearing in mind that $f$ is maximal, in order to get a contradiction it suffices to check that $\|g\|_{{FBL}^{(p)}[E]}=\|f\|_{{FBL}^{(p)}[E]}$.
To this end, take any $x_1^*,\dots,x_n^*\in \ell_{\infty}$ such that $\sum_{k=1}^n |x_k^*(a)|^p\leq 1$ for all $a\in A$. Let
$I$ be the set of those $k\in \{1,\dots,n\}$ such that $x_k^*\not\in R(y^*)$ and let $J:=\{1,\dots,n\}\setminus I$, so that for each $k\in J$
we have $x_k^*=\lambda_k y^*$ for some $\lambda_k>0$. Note that for every $a\in A$ we have
\begin{eqnarray*}
	\sum_{k\in I} |x_k^*(a)|^p+\sum_{k\in J}  |\lambda_k x^*(a)|^p&=&
	\sum_{k\in I} |x_k^*(a)|^p+|x^*(a)|^p\sum_{k\in J}  \lambda_k ^p \\ & \leq &
	\sum_{k\in I} |x_k^*(a)|^p+|y^*(a)|^p\sum_{k\in J}  \lambda_k^p\\
 &=	&\sum_{k=1}^n |x_k^*(a)|^p \leq 1.
\end{eqnarray*}
Hence,
$$
	\sum_{k=1}^n g(x_k^*)^p=
	\sum_{k\in I} f(x_k^*)^p+\sum_{k\in J}  f(\lambda_k x^*)^p
	\leq \|f\|_{{FBL}^{(p)}[\ell_1(A)]}^p.
$$
This shows that $\|g\|_{{FBL}^{(p)}[\ell_1(A)]} \leq \|f\|_{{FBL}^{(p)}[\ell_1(A)]}$, which is a contradiction.

(ii): Set $x^*:=(\sum_{k=1}^n (x^*_k)^p)^{\frac{1}{p}}$. By contradiction, suppose otherwise that $f(x^*)<(\sum_{k=1}^n (f(x^*_k))^p)^{\frac{1}{p}}$. 
Define $g:\ell_{\infty}\to \mathbb{R}_+$ by   
$$
	\begin{cases}
		g(\lambda x^*):=\lambda\ (\sum_{k=1}^n (f(x^*_k))^p)^{\frac{1}{p}} & \text{for all $\lambda>0$}, \\
		g(z^*):=f(z^*) & \text{for all $z^*\in \ell_{\infty} \setminus R(x^*)$}.
	\end{cases}
$$
Clearly, $g$ is positively homogeneous, $f\leq g$ and $f\neq g$. Again by the maximality of~$f$,
to get a contradiction it suffices to show that $\|g\|_{{FBL}^{(p)}[\ell_1(A)]}=\|f\|_{{FBL}^{(p)}[\ell_1(A)]}$.
Take $y_1^*,\dots,y_m^*\in \ell_{\infty}$ such that $\sum_{j=1}^m |y_j^*(a)|^p\leq 1$ for all $a\in A$. Let
$I$ denote the set of all $j\in \{1,\dots,m\}$ for which $y_j^*\not\in R(x^*)$ and let $J:=\{1,\dots,m\}\setminus I$, so that for each $j\in J$
we can write $y_j^*=\lambda_j x^*$ for some $\lambda_j>0$. Set $\mu:=\sum_{j\in J}\lambda_j^p$.
Since
\begin{eqnarray*}
	\sum_{j\in I}|y_j^*(a)|^p+\sum_{k=1}^n|(\mu)^{\frac{1}{p}} x_k^*(a)|^p&=& 	\sum_{j\in I}|y_j^*(a)|^p+\mu (x^*(a))^p \\ 
 &=&	\sum_{j\in I}|y_j^*(a)|+\sum_{j\in J}( \lambda_j x^*(a))^p\\
 &=& \sum_{j=1}^m|y_j^*(a)|^p \leq 1,
\end{eqnarray*}
for every $a\in A$, we obtain
\begin{eqnarray*}
	\sum_{j=1}^m g(y_j^*)^p&	=&\sum_{j\in I}f(y_j^*)^p + \sum_{j\in J}\lambda_j^p\Big(\sum_{k=1}^n f(x_k^*)^p\Big)
	\\ &=&\sum_{j\in I} f(y_j^*)^p+\sum_{k=1}^n f((\mu)^{\frac{1}{p}} x_k^*) \leq \|f\|_{{FBL}^{(p)}[\ell_1(A)]}^p.
\end{eqnarray*}
It follows that $\|g\|_{{FBL}^{(p)}[\ell_1(A)]} \leq \|f\|_{{FBL}^{(p)}[\ell_1(A)]}$, which is a contradiction.

(iii): We have $\|f\|_ {{FBL}^{(p)}[\ell_1(A)]}\geq \|f\|_\infty$. To prove the
equality, take finitely many $x_1^*,\dots,x_n^*\in \ell_{\infty}$
such that $\sum_{k=1}^n |x_k^*(a)|^p\leq 1$ for every $a\in A$. Then $x^*:=(\sum_{k=1}^n |x_k^*|^p)^{\frac{1}{p}}\in B_{\ell_{\infty}}$ and
$$
	\sum_{k=1}^n f(x^*_k)^p \stackrel{{\rm (i)}}{\leq}
	\sum_{k=1}^n f(|x^*_k|)^p \stackrel{{\rm (ii)}}{\leq} f(x^*)^p \leq \|f\|^p_\infty.
$$
This shows that $\|f\|_{FBL[\ell_1(A)]} \leq \|f\|_\infty$ and finishes the proof.
\end{proof}

\begin{lemma}\label{lem:normoflinear}
Let $A$ be a non-empty set and let $\phi:\ell_{\infty} \rightarrow \mathbb{R}$ be a linear functional. Define $g_\phi:\ell_{\infty} \to \mathbb{R}_+$ by
$$
	g_\phi(x^*) := |\phi(|x^*|^p)|^{\frac{1}{p}}
	\quad\mbox{ for all }x^*\in \ell_{\infty}. 
$$
Then $g_\phi\in H_p[\ell_1(A)]_+$ and 
$$
	\|g_\phi\|_{{FBL}^{(p)}[\ell_1(A)]}= \sup\big\{|\phi(x^*)|^{\frac{1}{p}}: \, x^*\in B_{\ell_{\infty}}\big\}.
$$
\end{lemma}
\begin{proof}
Clearly, $g_\phi$ is positively homogeneous. Take any $x^*_1,\ldots,x^*_n\in \ell_{\infty}$ such that 
$\sup_{a\in A} \sum_{k=1}^n|x^*_k(a)|^p\leq 1$. For each $k\in \{1,\dots,n\}$, let $\varepsilon_k \in \{-1,1\}$ be the sign of $\phi(|x^*_k|^p)$.
Then $\sum_{k=1}^n  \varepsilon_k|x^*_k|^p \in B_{\ell_{\infty}}$ and so
$$
	\sum_{k=1}^n g_\phi(x^*_k)^p= \sum_{k=1}^n \varepsilon_k \phi(|x^*_k|^p) 
	= \phi\left(\sum_{k=1}^n  \varepsilon_k|x^*_k|^p\right) \leq \sup\big\{|\phi(x^*)|: \, x^*\in B_{\ell_{\infty}}\big\}=:\alpha^p.
$$
This immediately shows that $\|g_\phi\|_{{FBL}^{(p)}[\ell_1(A)]}\leq \alpha$. 

For the converse, pick $x^*\in B_{\ell_{\infty}}$ and write it as $x^* = (x^*)_+ - (x^*)_-$, 
the difference of its positive and negative parts. Since $|((x^*)_+)^\frac{1}{p}(a)|^p + |((x^*)_-)^\frac{1}{p}(a)|^p=|(x^*)_+(a)| + |(x^*)_-(a)|=|x^*(a)|\leq 1$ for all $a\in A$, then we get
\begin{eqnarray*}
|\phi((x^*))| & \leq& |\phi(((x^*)_+)^\frac{1}{p})^p)| + |\phi(((x^*)_-)^\frac{1}{p})^p)| \\
&=& g_\phi(((x^*)_+)^\frac{1}{p})^p + g_\phi(((x^*)_-)^\frac{1}{p})^p \\
&\leq& \|g_\phi\|^p_{{FBL}^{(p)}[\ell_1(A)]}.
\end{eqnarray*}
 
This proves that $\alpha\leq \|g_\phi\|_{{FBL}^{(p)}[\ell_1(A)]}$.
\end{proof}

\begin{lemma}\label{lem:MaximalNecessaryCondition}
Let $A$ be a non-empty set and let $f \in H_p[\ell_1(A)]_+$ be maximal. 
Then there exists $\phi\in\ell_{\infty}$ such that $f=g_\phi$.
\end{lemma}
\begin{proof}
The case $f=0$ being trivial, we can suppose without loss of generality that $\|f\|_{{FBL}^{(p)}[\ell_1(A)]} = 1$. 
The set
$$
	C:= \{x^*\in (\ell_{\infty})_+ :\, f((x^*)^{\frac{1}{p}})>1\}
$$ 
is convex as a consequence of Lemma~\ref{lem:PropertiesMaximal}(ii). Let $U$ be the open
unit ball of~$\ell_{\infty}$. Since $\|f\|_\infty = \|f\|_{{FBL}^{(p)}[\ell_1(A)]} = 1$
(Lemma~\ref{lem:PropertiesMaximal}(iii)), we have $C\cap U=\emptyset$.
As an application of the Hahn-Banach separation theorem (cf. \cite[Proposition~2.13(ii)]{fab-ultimo}),
there is $\phi\in (\ell_{\infty})^*$ such that
\begin{equation}\label{eqn:HB}
	\phi(y^*) < \inf \{\phi(x^*) :\,  x^* \in C\}
	\quad\mbox{for all }y^*\in U.
\end{equation}
We can suppose that $\|\phi\|=1$ and so $\|g_\phi\|_{FBL[\ell_1(A)]}=1$ (Lemma~\ref{lem:normoflinear}). 

We claim that $f=g_\phi$. Indeed, since $f$ is maximal, it is enough to show that $f(x^*)\leq g_\phi(x^*)$ for every
$x^*\in \ell_{\infty}$ with $f(x^*)>0$. Fix $t>1$. By Lemma~\ref{lem:PropertiesMaximal}(i), we have
$f(|x^*|)\geq f(x^*)>0$ and so
$$
	f\left(\frac{t}{f(|x^*|)}|x^*|\right) = t > 1.
$$ 
Therefore $\left(\frac{t}{f(|x^*|)}|x^*|\right)^p \in C$ and \eqref{eqn:HB} yields
$$
	\phi\Big(\left(\frac{t}{f(|x^*|)}|x^*|\right)^p \Big)\geq \sup \{|\phi(y^*)|:\, y^*\in U\}=\|\phi\|=1.
$$
We conclude that $(t\phi(|x^*|))^p \geq f(|x^*|)^p $ for any $t>1$, so $(\phi(|x^*|)^p \geq f^p(|x^*|) \geq f(x^*)^p$ and then $g_\phi(|x^*|)=(\phi(|x^*|^p))^{\frac1p} \geq f(x^*)$.
The proof is complete.
\end{proof}

Given any non-empty set~$A$, it is well-known that every $\phi\in \ell_{\infty}^*$
can be written in a unique way as $\phi = \phi_0 + \phi_1$, where 
\begin{itemize}
\item $\phi_0 \in \ell_1(A)$ (identified as a subspace of~$\ell_{\infty}^*$), 
\item $\phi_1 \in \ell_{\infty}^*$ vanishes on all finitely supported elements of~$\ell_{\infty}$.
\end{itemize}
Moreover, $\|\phi\| = \|\phi_0\| + \|\phi_1\|$. 

\begin{lemma}\label{lem:CountableSum}
Let $A$ be a non-empty set and $\phi\in \ell_1(A)$. Then $g_\phi \in FBL^p[\ell_1(A)]$.
\end{lemma}

\begin{proof} Let $(e_a)_{a\in A}$ be the unit vector basis of~$\ell_1(A)$. Set $B=supp(\phi)$ and note that for $x^*\in\ell_\infty(A)$
$$g_\phi(x^*)=\Big|\sum_{a\in B}\phi(a)|x^*(a)|^p\Big|^{\frac1p}.$$
For every finite subset $S\subset B$, we have that 
$$
f_{S}:=\Big|\sum_{a\in S} \phi(a)|\delta_a|^p\Big|^{\frac1p}\in FBL^p[\ell_1(A)].
$$

Let $x_1^*,\ldots,x_n^*\subset \ell_\infty$ such that $\sup_{a\in A} \sum_{j=1}^n |x_j^*(a)|^p\leq1$. We have that
\begin{eqnarray*}
\sum_{j=1}^n |g_{\phi}(x_j^*)-f_{S}(x_j^*)|^p&=&\sum_{j=1}^n\Big|\Big|\sum_{a\in B}\phi(a)|x_j^*(a)|^p\Big|^{\frac1p}- \Big|\sum_{a\in S} \phi(a)|x_j^*(a)|^p \Big|^{\frac1p} \Big|^p\\
&\leq& \sum_{j=1}^n\Big|\sum_{a\in B\backslash S}\phi(a)|x_j^*(a)|^p\Big|\\
&\leq& \sum_{a\in B\backslash S}|\phi(a)|.
\end{eqnarray*}
Hence, $\|g_\phi - f_{S}\|_{FBL^{(p)}[\ell_1(A)]}\leq \Big(\sum_{a\in B\backslash S}|\phi(a)|\Big)^{\frac1p}$. Since $\phi\in\ell_1(A)$ for $\varepsilon >0$ we can take finite $S\subset B$ such that $\sum_{a\in B\backslash S} |\phi(a)|<\varepsilon$. Therefore, we have $g_\phi\in FBL^p[\ell_1(A)]$.\end{proof}

\begin{lemma}\label{lem:continuity}
Let $A$ be a non-empty set, $\xi: B_{\ell_{\infty}} \to \mathbb{R}$ a $w^*$-continuous function and $\phi\in \ell_{\infty}^*$. 
If $\xi \leq g_\phi$ on~$B_{\ell_{\infty}}$, then $\xi \leq g_{\phi_0}$ on~$B_{\ell_{\infty}}$ as well.
\end{lemma}
\begin{proof} The $w^*$-topology on~$B_{\ell_{\infty}}=[-1,1]^A$ agrees with the pointwise topology.
Since the map $x^* \mapsto |x^*|^p$ is $w^*$-continuous when restricted to~$B_{\ell_{\infty}}$ and $\phi_0$ is $w^*$-continuous, 
we have that $g_{\phi_0}$ is $w^*$-continuous on~$B_{\ell_{\infty}}$.
On the other hand, if $x^*\in B_{\ell_{\infty}}$ is finitely supported, then $\phi_1(|x^*|^p) = 0$ and, therefore, we have $\xi(x^*) \leq g_{\phi}(x^*) 
= |\phi_0(|x^*|^p)|^{\frac{1}{p}}=g_{\phi_0}(x^*)$. 
Since the finitely supported elements of~$B_{\ell_{\infty}}$ are $w^*$-dense and the functions $\xi$ and $g_{\phi_0}$ are $w^*$-continuous on~$B_{\ell_{\infty}}$, 
we conclude that $\xi \leq g_{\phi_0}$ on~$B_{\ell_{\infty}}$.
\end{proof}

We arrive at the main result of this section:

\begin{theorem}\label{thm:Nakano}
The norm of $FBL^p[\ell_1(A)]$ has the strong Nakano property for any non-empty set~$A$.
\end{theorem}
\begin{proof}
Let $ A \subset FBL^p[\ell_1(A)]_+$ be an upwards directed family such that 
$$
	\sup\{\|f\|_{FBL^p[\ell_1(A)]} : \, f\in A\} = 1.
$$ 
We are going to show that $A$ has an upper bound of norm 1. 

Note that $A$ is pointwise bounded (because $
	\|f\|_\infty:=\{|f(x^*)|: \, x^*\in B_{E^*}\big\} \leq \|f\|_{{FBL}^{(p)}[E]}$) and let $h:\ell_{\infty}\to \mathbb{R}_+$ be defined
as $h(x^*):=\sup\{f(x^*):f\in  A\}$ for all $x^*\in \ell_{\infty}$. 
Lemma~\ref{lem:PointwiseSupremum} ensures that $h\in H_p[\ell_1(A)]_+$ and $\|h\|_{FBL^p[\ell_1(A)]} = 1$. 
Now let $g\in H_p[\ell_1(A)]_+$ be maximal such that 
$g\geq h$  and $\|g\|_{FBL^p[\ell_1(A)]}=1$ (apply Lemma~\ref{lem:Zorn}).
Then $g=g_\phi$ for some $\phi \in \ell_{\infty}$ with $\|\phi\|=1$
(combine Lemmas~\ref{lem:normoflinear} and~\ref{lem:MaximalNecessaryCondition}).

Given any $f\in  A \subset FBL^p[\ell_1(A)]$, we have $f\leq g_\phi$ and the restriction $f|_{B_{\ell_{\infty}}}$ is $w^*$-continuous 
(cf. \cite[Remark 2.5]{OTTT}), hence Lemma~\ref{lem:continuity} yields $f(x^*)\leq g_{\phi_0}(x^*)$ for every $x^*\in \ell_{\infty}$ (bear in mind that
both $f$ and $g_{\phi_0}$ are positively homogeneous). Since $g_{\phi_0} \in FBL^p[\ell_1(A)]$
(Lemma~\ref{lem:CountableSum}) and 
$$1
=\|\phi\| \geq \|\phi_0\|\geq\|g_{\phi_0}\|^p_{FBL[\ell_1(A)]}
$$
(by Lemma~\ref{lem:normoflinear}), it turns out that $g_{\phi_0}$ is the upper bound of~$A$ in $FBL^p[\ell_1(A)]$ that we were looking for. 
The proof is finished.
\end{proof}

\section*{Acknowledgements}

This research has been funded by CSIC i-COOP program under grant COOPB20617. Research of P. Tradacete partially supported by grants PID2020-116398GB-I00 and CEX2019-000904-S funded by MCIN/AEI/10.13039/501100011033, as well as by a 2022 Leonardo Grant for Researchers and Cultural Creators, BBVA Foundation.

\end{document}